\documentclass{amsart}
\usepackage{amsfonts,amsmath,amssymb,amscd,latexsym,graphicx,
epsfig,color,a4wide}

\title{Dual Garside structure and reducibility of braids}
\theoremstyle{plain}

\newtheorem{defi}{Definition}
\newtheorem{prop}[defi]{Proposition}

\newtheorem{lem}[defi]{Lemma}
\newtheorem{thm}[defi]{Theorem}
\newtheorem{claim}[defi]{Claim}
\newtheorem{ex}[defi]{Example}
\newtheorem{remark}[defi]{Remark}

\newcommand{\cvr}{\curvearrowright}
\newcommand{\cvl}{\curvearrowleft}
\newcommand{\clr}{\circlearrowright}
\newcommand{\cll}{\circlearrowleft}
\begin{document}

\author{Matthieu Calvez}\thanks{This work was done while the author was staying at the University of Seville. The visit was partially supported by a grant for international mobility by the Universit\'{e} Europ\'{e}enne de Bretagne, by the Spanish research group FQM-218 (Junta de Andaluc\'{i}a) and by the Spanish Project MTM2010-19355.}

\begin{abstract} Benardete, Gutierrez and Nitecki showed an important result which relates the geometrical properties of a braid, as a homeomorphism of the punctured disk, to its algebraic Garside-theoretical properties. Namely, they showed that if a braid sends a {\it standard} curve to another {\it standard} curve, then the image of this curve 
under the action of each factor 
of the left normal form of the braid (with the classical Garside structure) is also standard.
We provide a new simple, geometric proof of the result by Benardete-Gutierrez-Nitecki, which can be easily adapted to the case of the dual Garside structure of braid groups, with the appropriate definition of standard curves in the dual setting. This yields a new algorithm for determining the Nielsen-Thurston type of braids.
\end{abstract}

\maketitle
\section{Introduction}
Braid groups are both mapping class groups and Garside groups.
The links between these two features of braid groups seem
to be very deep and their investigation is currently the
subject of much research.
In particular the Garside-theoretic
approach to the problem of deciding algorithmically
the Nielsen-Thurston type of a given braid turns
out to be very fruitful
\cite{bgn},\cite{leelee},\cite{bertjuan},\cite{juan2010},\cite{bertmatth}.

The $n$-strand braid group is naturally identified
with the mapping class group of the $n$-times punctured
disk $D_n$. Braids induce a (right) action on the set
of isotopy classes of simple closed curves in $D_n$:
considering the isotopy class $[\mathcal C]$
of a simple closed curve $\mathcal C$ and an $n$-braid $x$,
the isotopy class of simple closed curves resulting
from the action of $x$ on $[\mathcal C]$ will be denoted
by $[\mathcal C]^x$. The simple closed curves
we shall consider in the present paper will be
\emph{nondegenerate}, that is surrounding
more than one and less than $n$ punctures.

Let us assume for the moment that
$D_n$ is parametrized as the disk with
diameter $[0,n+1]$ in $\mathbb C$ and
points $1,2,\ldots,n$ removed.
In this setting, a curve $\mathcal C$ is
said to be {\it standard}, or {\it round},
if it is isotopic to a geometric circle in $D_n$.
That is, if the punctures enclosed by $\mathcal C$
are consecutive. An isotopy class~$[\mathcal C]$
is said to be standard, or round,
if some (hence every) representative is round.
Round curves are particularly useful for
decomposing a reducible braid into its
corresponding components \cite{juan2010}.
As every reducible braid has a conjugate
which preserves a family of round curves~\cite{bgn},
searching for such a conjugate becomes
a possible strategy both for determining whether
a braid is reducible and for finding its geometric components.
Benardete, Gutierrez and Nitecki~\cite{bgn}
explain how to determine whether such a conjugate exists,
and also how to find it, thanks to the
classical Garside structure of braid groups.

Braid groups are the main examples of
Garside groups~\cite{dehparis}.
This means that they admit a lattice structure,
and a special element denoted $\Delta$,
satisfying some properties first discovered by
Garside in \cite{garside}. We will refer to this as
the {\it classical Garside structure} of the braid group.
Using this structure, one can define the
{\it left normal form} of a braid $x$,
which is a unique decomposition of the form
$x=\Delta^p x_1\cdots x_r$
(see~\cite{echlpt},\cite{elrifaimorton})
in which the factors belong to
the set of the so-called {\it simple elements}.

The result by Benardete, Gutierrez and Nitecki,
which relates round curves and left normal forms,
is the following:
\begin{thm}[\cite{bgn},\cite{leelee}] \label{theoremartin}Let $\mathcal C$ be a standard curve in $D_n$. Let $x=\Delta^p x_1\cdots x_r$ be a braid in (classical) left normal form. If $[\mathcal C]^x$ is standard, then $[\mathcal C]^{\Delta^p x_1\cdots x_m}$ is standard for $m=1,\ldots,r$.
\end{thm}
Thanks to this result, and applying
special conjugations called {\it cyclings} and
{\it decyclings}, it is possible to determine
whether a braid has a conjugate which preserves
a family of round curves, hence it is possible
to know whether a braid is reducible,
with the aid of the Garside structure of the braid group~\cite{bgn}.

The method mentioned in the previous paragraph
requires the computation of 
a big subset of the
conjugacy class of a braid. This
has been improved in~\cite{bertjuan},
avoiding the computation of such a big subset,
at the cost of enlarging the set of standard curves
to include both {\it round} and 
{\it almost-round} curves.
This raises the question of whether the notion of round curves,
and the use of the classical Garside structure,
are the best choices for this kind of techniques.

There is another well-known Garside structure
of the braid group, discovered by Birman, Ko and Lee \cite{bkl},
which is known as the {\it dual Garside structure}.
With respect to this structure, the
left normal form of a braid is a unique decomposition
$x=\delta^p x_1\cdots x_r$, where $\delta$ is
the special (Garside) element, and the factors are simple,
with respect to the dual structure.

In this paper we prove the analogue of
Theorem \ref{theoremartin} in the dual setting.
We remark that the proofs of Theorem
\ref{theoremartin} given in \cite{bgn} and
\cite{leelee} cannot be adapted in a natural way
to the dual Garside structure of the braid group.
For that reason we give a new proof of Theorem
\ref{theoremartin}, in the classical setting,
which can be naturally adapted to the dual setting.
To this end we also introduce a natural notion
of standard curve related to the dual Garside
structure of the braid group. Namely, as round curves determine standard parabolic subgroups of the braid group, with the Artin strucure,  the standard curves in the dual setting will be those determining standard parabolic subgroups, with the dual structure.

It is important to mention that a generalization of Theorem \ref{theoremartin} to
Artin-Tits groups of spherical type,
with the classical Garside structure, is given in \cite{godelle}.
The powerful algebraic methods used in \cite{godelle} are based in the theory of Garside categories and seem to allow further generalization of Theorem \ref{theoremartin} to dual Garside structures \cite{personalgodelle},
although this does not appear in the literature. Here we present simple, geometric proofs, for the particular but important cases of the two well known structures of the braid groups.

Our proof of Theorem \ref{theoremartin}
can be sketched as follows. We show that
if $\mathcal C$ is round and~$[\mathcal C]^{\Delta^px_1}$
is not round, then $[\mathcal C]^{\Delta^px_1\cdots x_m}$
is not round for $m=1,\ldots,r$.
Hence $[\mathcal C]^{\Delta^px_1\cdots x_r}$ cannot be round,
contradicting the hypothesis.

We assume the usual Artin generator $\sigma_i$,
for $i=1,\ldots,n-1$, to be the counterclockwise
half Dehn twist along the segment $[i,i+1]$.
We will see that if $[\mathcal C]^{\Delta^px_1}$ is not round,
then a portion of (a suitable representative of)
$[\mathcal C]^{\Delta^px_1}$ crosses the real line in the
way shown in Figure 1, for some $i<j<k$.
Moreover, $x_1\sigma_j$ is simple, meaning that
the strands of $x_1$ ending at $j$ and $j+1$
do not cross.

\begin{figure}[ht]
\centerline{\includegraphics{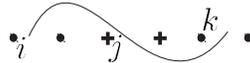}}
\caption{Strands which arrive at positions $j$ and $j+1$ (depicted as crosses) have not crossed in $x_1$.}
\label{imageintro}
\end{figure}
We then prove by recursion that,
as $x_1\cdots x_r$ is in left normal form,
the above properties must hold for
$[\mathcal C]^{\Delta^px_1\cdots x_m}$, for $m=1,\ldots,r$.
That is, some portion of (a suitable representative of)
$[\mathcal C]^{\Delta^px_1\cdots x_m}$ must
cross the real line as in Figure \ref{imageintro},
for some $i_m<j_m<k_m$, where $x_m\sigma_{j_m}$
is simple. This implies in particular that
$[\mathcal C]^{\Delta^px_1\cdots x_r}$ is not round,
showing Theorem \ref{theoremartin}.

This proof can be adapted to the
dual setting as follows. First,
in order to work with the dual Garside structure
of the braid group, it is more convenient
to parametrize~$D_n$ as the unit disk in
$\mathbb C$ with set of punctures
$\{\frac{1}{2}e^{\frac{-2 i k \pi}{n}}, \ k=1,\ldots,n\}$.
Throughout the paper the puncture
$\frac{1}{2}e^{\frac{-2 i k\pi}{n}}$
will be denoted $k$ for brevity,
and the disk $D_n$ with this parametrization
will be denoted $D_n^*$.
We remark that, if one defines standard curves
as isotopy classes of geometric circles, as above,
then the analogue of Theorem \ref{theoremartin}
is not true in the dual setting
(see Example \ref{exampleprime}).
Hence we need a different definition
for standard curves, adapted to the
dual Garside structure:
\begin{defi}[See also Definition \ref{defistandard}]
\label{informaldefistandard}
A simple closed curve in $D_n^*$ is called {\em standard} if it is isotopic to a curve which can be expressed, in polar coordinates, as a function $\rho=\rho(\theta)$, for $\theta \in [0,2\pi[ $. See Figure \ref{imageintrobkl}(a). An isotopy class $[\mathcal C]$ is said to be standard if some (hence every) representative is standard.
\end{defi}
The main result of this paper is the following:
\begin{thm}
\label{theorembkl} Let $\mathcal C$ be a standard curve in $D_n^*$. Let $x=\delta^p x_1\cdots x_r$ be a braid in dual left normal form. If $[\mathcal C]^x$ is standard, then $[\mathcal C]^{\delta^p x_1\cdots x_m}$ is standard for $m=1,\ldots,r$.
\end{thm}

\begin{figure}[ht]
\centerline{\includegraphics{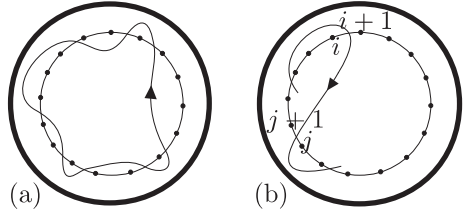}}
\caption{(a) shows a standard curve as in Definition \ref{informaldefistandard}. (b) A part of a nonstandard curve.}
\label{imageintrobkl}
\end{figure}
The proof of this result parallels
the proof of Theorem \ref{theoremartin}
and is sufficiently simple to allow an outline here.
First recall that the usual generators of the
dual Garside structure are the braids
$a_{i,j}=a_{j,i}$ corresponding to the
counterclockwise half Dehn twist along
the chord segment joining the punctures
$i$ and $j$, for each pair ${1\leqslant i,j\leqslant n}$, with $i\neq j$.

We will see that if $[\mathcal C]^{\delta^px_1}$
is not standard, then a portion of
(a suitable representative of)
$[\mathcal C]^{\delta^px_1}$ crosses the circle
of radius $\frac{1}{2}$ in the way shown
in Figure \ref{imageintrobkl}(b),
for some~${1\leqslant i,j\leqslant n}$.
Moreover, $x_1a_{i,j}$ is simple.
We then prove by recursion that, as
$x_1\cdots x_r$ is in left normal form,
the above properties must hold for
$[\mathcal C]^{\delta^px_1\cdots x_m}$, for
$m=1,\ldots,r$. That is, some portion of
(a suitable representative of)
$[\mathcal C]^{\delta^px_1\cdots x_m}$ must cross
the circle of radius $\frac{1}{2}$ as
in Figure \ref{imageintrobkl}(b),
for some ${1\leqslant i_m,j_m\leqslant n}$,
where $x_m a_{i_m,j_m}$ is simple.
This implies in particular that
$[\mathcal C]^{\delta^px_1\cdots x_r}$
is not standard, showing Theorem 3.

The following two sections contain the
detailed proofs of Theorems \ref{theoremartin}
and \ref{theorembkl}, respectively. The last section contains a new algorithm for solving the reducibility problem in braid groups, based on Theorem \ref{theorembkl}.

{\bf{Acknowledgements.}} The author thanks Eddy Godelle for his receptiveness to questions about his papers on parabolic subgroups of Garside groups. The author also wishes to thank Juan Gonz\'{a}lez-Meneses for countless helpful comments on earlier versions of the paper.

\section{The classical Artin-Garside case}\label{artinproof}
This section deals with the classical case.
Throughout the section, we thus assume~$D_n$
to be parametrized as the disk with diameter
$[0,n+1]$ in $\mathbb C$ and with the
points~${1,2,\ldots,n}$
removed.

In order to make rigorous the idea
expressed in Figure \ref{imageintro},
we associate to each isotopy class of
simple closed curves~$[\mathcal C]$ in $D_n$
a unique \emph{reduced} word $W([\mathcal C])$
which we shall define in Subsection~\ref{fromcurvestowords}.
The word $W([\mathcal C])$ allows
a careful description of the action of
positive braids on $[\mathcal C]$;
this will be completed in Subsection
\ref{actionofbraids}.
Subsection~\ref{prooftheoremartin} will
be devoted to the proof of Theorem~\ref{theoremartin}.

\subsection{From curves to words}\label{fromcurvestowords}
We will always assume
that the curves under consideration
are nondegenerate, simple and closed,
so unless otherwise stated,
the word ``curve" alone will mean
``nondegenerate simple closed curve". \\
Let $\mathcal C$ be a curve in $D_n$,
and suppose that it 
crosses the real line a finite number of times, 
and all these crossings are transverse. 
We shall associate
to $\mathcal C$ a word $W(\mathcal C)$
that we now define.

Our strategy to define the word
$W(\mathcal C)$ already mostly appears
in \cite{fgrrw}, Appendix~A.
We proceed as follows.
Choose
a point $\ast$ of $\mathcal C$ which
lies on the real line as well as an
orientation for $\mathcal C$.
Running along $\mathcal C$ following
the chosen orientation,
starting and ending at $\ast$, determines
a word in the
alphabet
$X=\{\smallsmile,\smallfrown,0,\ldots,n\}$
as follows:
each arc through the upper half
plane contributes a letter $\smallfrown$
to the word, each
arc through the lower half plane,
a letter $\smallsmile$,
and each intersection with the
segment~${]i,i+1[}$ yields the letter $i$.

The number corresponding to $\ast$ can be
chosen to be either at the beginning
or at the end of the word.
The word obtained in this way will be denoted
$W(\mathcal C)$ and we call it the
\emph{word associated to $\mathcal C$}.
Choosing another point $\ast$ or putting
the letter determined by the intersection
point $\ast$ at the beginning or at the end
of $W (\mathcal C)$ corresponds to a cyclic
permutation of the letters in $W(\mathcal C)$,
whereas choosing the reverse orientation of
$\mathcal C$ yields the reverse of
$W(\mathcal C)$. Hence, words associated
to curves are to be considered up to
cyclic permutation of their letters
and up to reverse.
\begin{ex}\label{words}\rm
Let $\mathcal C$ be the curve depicted in Figure \ref{curvexampleartin}.
\begin{figure}[ht]
\centerline{\includegraphics{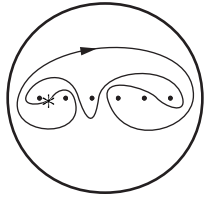}}
\caption{The curve $\mathcal C$ of Example \ref{words}.}
\label{curvexampleartin}
\end{figure}
If we fix $\ast$ in the interval~${]1,2[}$ and the
clockwise orientation, then the word associated to $\mathcal C$ is
$$W(\mathcal C)=1\smallfrown 2\smallsmile 0\smallfrown 6
\smallsmile 3\smallfrown 5\smallsmile 6\smallfrown 3
\smallsmile 2\smallfrown 0\smallsmile.$$
\end{ex}
Notice that two curves related
by an isotopy of $D_n$ fixing the real
diameter setwise have the same associated word.

We say that the word associated to
$\mathcal C$  is \emph{reduced} if it
does not contain any subword of the form
$i\smallsmile i$ or $i\smallfrown i$.
We say that a curve $\mathcal C$ is
\emph{reduced} if its associated
word~$W(\mathcal C)$ is reduced.

Notice that reduced curves are exactly
those which do not bound any bigon
(see~\cite{fgrrw}) together with the real line.
According to \cite{fgrrw}, every curve
$\mathcal C$ is isotopic to a reduced one
$\mathcal C^{red}$, which is unique up to
isotopy of $D_n$ fixing the real diameter setwise.
We thus may finally define, for each isotopy
class of curves $[\mathcal C]$ in $D_n$,
its associated reduced word as
$W([\mathcal C])=W(\mathcal C^{red})$.

\subsection{The action of positive braids}\label{actionofbraids}
Let $B_n^{+}$ be the submonoid of $B_n$
generated by $\sigma_1,\ldots,\sigma_{n-1}$,
called the monoid of positive braids.
Instead of the usual Artin generators,
it will be convenient to work with
the following bigger generating set of $B_n^{+}$:
$$\{\Sigma_{p,k}=\sigma_{p-1}\ldots\sigma_{p-k},\ \ 1\leqslant k<p\leqslant n\}.$$

According to our previous convention
that $\sigma_i$ is a counterclockwise
half-Dehn twist along the segment $[i,i+1]$,
the braid $\Sigma_{p,k}$ corresponds
to a move of the puncture numbered $p$ through
the upper half plane up to the position $p-k$ while
the punctures~${p-1,p-2,\ldots,p-k}$
are shifted one position to the right.

Given the isotopy class of a
curve $\mathcal C$ and a positive
braid $x$ we are going to describe
which transformations have to be
performed on the word $W=W([\mathcal C])$
in order to obtain the word
$W([\mathcal C]^x)$. To this purpose
we first focus on the case $x=\Sigma_{p,k}$.

Since the generators we are considering
are mainly moves of punctures in the upper
half plane, we assume that the action of the braids
$\Sigma_{p,k}$
mainly modifies the upper arcs
occuring in~$\mathcal C^{red}$
(from which
new lower arcs can arise) while the
lower arcs are only modified by
translating their endpoints.
This will be described by the following formulae, which are also
depicted in figure~\ref{figureformulae}.

We define, for $i<j$:

$$
(i\smallfrown j)^{\Sigma_{p,k}}=
\left\{
\begin{array}{llr}
i\smallfrown j & \text{if $[p-k,p[\ \cap\ \{i,j\}=\emptyset$} & (\text{$F1$}) \\
(i+1)\smallfrown (j+1) & \text{if $[p-k,p[\ \cap\ \{i+1,j\}=\{i+1,j\}$} & (\text{$F2$})\\
(i+1)\smallfrown (p-k) \smallsmile (p-k-1)\smallfrown j &
\text{if $[p-k,p[\ \cap\ \{i,j\}=\{i\}$} & (\text{$F3$})\\
i\smallfrown (p-k-1)\smallsmile(p-k)\smallfrown (j+1) & \text{if $[p-k,p[\ \cap\ \{i+1,j\}=\{j\}$} & (\text{$F4$})
\end{array}
\right.
$$
We can define in the same way
$(j\smallfrown i)^{\Sigma_{p,k}}$.
It suffices to take the reverse of the
above formulae (the picture is exactly the same than in
Figure \ref{figureformulae}).

\begin{figure}[ht]
\centerline{\includegraphics{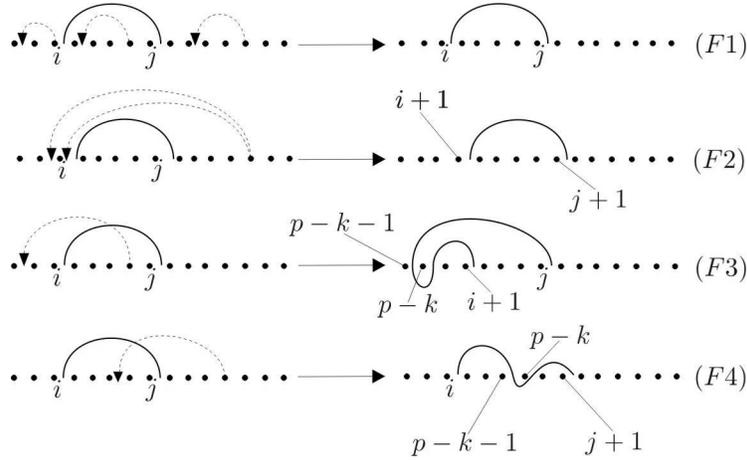}}
\caption{How the action of the braid $\Sigma_{p,k}$ does transform
upper arcs? In dashed lines is represented the trace of the move of
the puncture initially numbered $p$ (up to the position $p-k$).
In continuous lines the arc $i\smallfrown j$ on the left hand side and its image $(i\smallfrown j)^{\Sigma_{p,k}}$ on the right hand side.}
\label{figureformulae}
\end{figure}

Now, let $\widetilde{W}$ be the word
obtained by replacing each subword
$i\smallfrown j$ in $W$ by the corresponding
subword $(i\smallfrown j)^{\Sigma_{p,k}}$.
This transforms the lower arcs
$(i\smallsmile j)$ in $W$, just by
translating their endpoints.

Notice that $\widetilde W$ is not
necessarily reduced, so that
$W([\mathcal C]^{\Sigma_{p,k}})$ and
$\widetilde W$ are possibly not the same.
The following explains how to turn
$\widetilde W$ into the reduced word
$W([\mathcal C]^{\Sigma_{p,k}})$:

\begin{lem}
\label{reductionartin}
Let $[\mathcal C]$ be an isotopy class of curves and
$W=W([\mathcal C])$.
Let ${1\leqslant k<p\leqslant n}$.
Let $\widetilde{W}$ be as above, and let
$W^{\Sigma_{p,k}}$ be the word obtained
from~$\widetilde{W}$ by removing all
instances of subwords of the form
$p\smallsmile p\smallfrown$.
Then~${W^{\Sigma_{p,k}}=W([\mathcal C]^{\Sigma_{p,k}})}$.
\end{lem}

\begin{proof}
We observe that the formulae
defining $(i\smallfrown j)^
{\Sigma_{p,k}}$
do not contain any subword of the
form $c\smallsmile c$ nor $c\smallfrown c$.
Thus the possible instance of such
a subword in $\widetilde{W}$ necessarily
arises from the transformation of a lower arc.
Notice that a lower arc~$c\smallsmile d$
can only be transformed into an
arc $k\smallsmile l$, with $k\in\{c,c+1\}$
and $l\in\{d,d+1\}$.
Hence the latter arc forms a bigon with the
horizontal axis if and only if $c+1=d$,
$k=c+1$ and $l=d$ (up to reverse we may
suppose that $c<d$),
in which case $\widetilde{W}$ contains the
subword $d\smallsmile d$.
By the formulae defining
$(i\smallfrown j)^{\Sigma_{p,k}}$, this happens
if and only if $d=p$, that is
$c\smallsmile d=(p-1)\smallsmile p$.
In particular we have shown that no
subword of the form $c\smallfrown c$
can arise in $\widetilde{W}$.\\
We now claim that removing the
subwords $p\smallsmile p\smallfrown$
is sufficient in order to turn~$\widetilde{W}$
into a reduced word; that is, every
sequence $p\smallsmile p\smallfrown$ in
$\widetilde{W}$ is a subsequence of a larger one,
of the form $a\smallfrown p\smallsmile p\smallfrown b$,
with $a\neq b$.\\
Let $h\smallfrown (p-1)\smallsmile p\smallfrown l$ be a subword
of $W$.
As $\mathcal C$ is a simple curve,
$h,l$ must be in one of the following
three cases (see Figure \ref{figurelemmared}):
\begin{itemize}
\item[1)] $h<p-1<p<l$,
\item[2)] $l\leqslant h<p-1<p$,
\item[3)] $p-1<p<l\leqslant h$.
\end{itemize}
We shall show that in the three cases, the
subword $a\smallfrown p\smallsmile p\smallfrown b$
of $\widetilde{W}$ yielded by the action
of $\Sigma_{p,k}$ satisfies
$$a\in\{p-k,\ldots,p-1\}$$
and
$$b\in\{0,\ldots,p-k-1\}\cup\{p+1,\ldots,n\},$$
and thus $a\neq b$, as claimed.

\begin{figure}[ht]
\centerline{\includegraphics{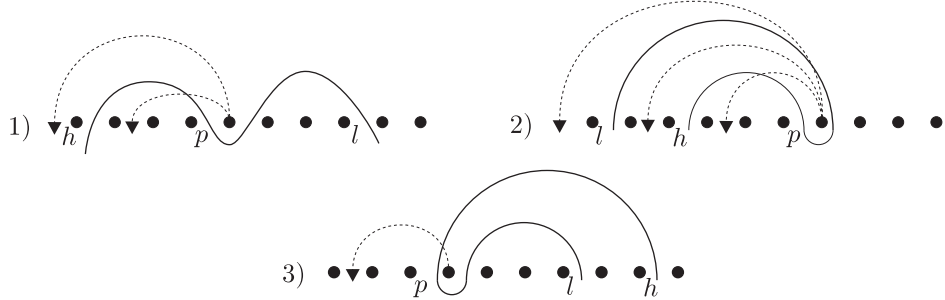}}
\caption{The possible cases of subwords $h\smallfrown (p-1)\smallsmile p\smallfrown l$ appearing in $W$, with the distinct possible moves
of the point $p$. }
\label{figurelemmared}
\end{figure}
In the first two cases the arc $(h\smallfrown p)$
will yield
$a=h+1$ or $a=p-k$,
satisfying $a\in \{p-k,\ldots,p-1\}$.
In the third case we will also obtain $a=p-k$.
On the other hand, in the first and
third cases we have $b=l$ (hence $b>p$);
in the second case, either $b=l$ if $p-k\geqslant l+1$
or $b=p-k-1$ otherwise (and $b<p-k$).

This completes the proof of the lemma.
\end{proof}

We thus can associate to each isotopy class
of curves $[\mathcal C]$ and each braid $\Sigma_{p,k}$
the word $W([\mathcal C])^{\Sigma_{p,k}}$
defined thanks to the above construction.
We are now able to define, for each
isotopy class of curves $[\mathcal C]$, the
image of its associated reduced
word $W=W([\mathcal C])$ under the action
of some positive braid $x$.
Indeed, if $x$ is expressed as a product
$x=\prod_{m=1}^{r} \Sigma_{p_m,k_m}$,
then by Lemma \ref{reductionartin},
the inductive formula
$$W^x=
({W^{\prod_{m=1}^{r-1} \Sigma_{p_m,k_m}}})^{\Sigma_{p_r,k_r}}$$
defines a word on $X$ which is the reduced
word associated to
$[\mathcal C]^x$ (hence
it does not depend on
the chosen decomposition of $x$ in terms
of braids $\Sigma_{p,k}$).
This can be written
${W({[\mathcal C]}^x)=W([\mathcal C])^x}$.

\subsection{Proof of Theorem \ref{theoremartin}}\label{prooftheoremartin}
By abuse
of notation, instead of speaking about
isotopy classes of simple closed
curves, we will simply
speak about curves, meaning that we
are considering the reduced
representatives. Consequently,
the letter $\mathcal C$ will denote the reduced
representative of
the isotopy class of the curve $\mathcal C$
and its associated reduced word
will be denoted $W(\mathcal C)$.

We introduce a class of curves which
is larger than the class of round curves:
\begin{defi}\label{D:presquerond}
A curve $\mathcal C$ will be
called \emph{almost-round}
if the word $W(\mathcal C)$ (up
to cyclic permutation of its letters
and reversing) can be
written as $W(\mathcal C)=w_1w_2$, where
the arcs in $w_1$
are oriented from left to right whereas
those in $w_2$ are oriented from right to
left. (See Figure \ref{F:examplepresquerond}).
\end{defi}

\begin{figure}[htb]
\centerline{\includegraphics{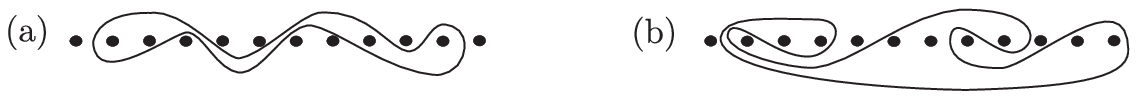}}
\caption{The curve in Part (a) is almost-round since its associated reduced word can be written as
$$W=w_1w_2=(1\smallfrown 4\smallsmile 6\smallfrown9\smallsmile 10\smallfrown11)(\smallsmile8\smallfrown6\smallsmile4\smallfrown3\smallsmile).$$ The curve in Part (b) is not almost-round.}
\label{F:examplepresquerond}
\end{figure}
The reduced words associated to almost-round
curves satisfy the following necessary condition:

\begin{lem}
\label{lemma2}
Let $\mathcal C$ be an almost-round, not round curve. Then
$W(\mathcal C)$ (up to cyclic permutation
and reversing) must contain
a subword of the form
$i\smallfrown j\smallsmile l$
for some $0\leqslant i<j<l\leqslant n$.

\end{lem}
\begin{proof}
The curve $\mathcal C$ has a unique local minimum
(and a unique local maximum) in the horizontal direction.
Let $\ast$ be this local minimum.
Choose the clockwise orientation for $\mathcal C$,
and notice that the arcs oriented
from left to right starting at~$\ast$
form a subword $i_1\smallfrown i_2\ldots\smallfrown i_k$
(ending with an upper arc, as $\mathcal C$ is simple).
If~${k>}2$, then $i_1\smallfrown i_2\smallsmile i_3$ satisfies
the required hypothesis.
If $k=2$, choose the counterclockwise orientation.
As $\mathcal C$ is not round, $W(\mathcal C)$ must
start with a
subword $i'_1\smallsmile i'_2\ldots \smallsmile i'_k$
with $k\geqslant 4$
and $i'_2\smallfrown i'_3\smallsmile i'_4$
does the job.
\end{proof}
Before proving Theorem \ref{theoremartin},
we introduce some more notation. We will
say that some $0<j<n$ is a
\emph{bending point for (or bends)}
a curve $\mathcal C$ if the reduced
word $W(\mathcal C)$ admits
(up to cyclic permutation of its letters
and up to reverse) a subword of the
form $i\smallfrown j\smallsmile l$,
for some $0\leqslant i<j<l\leqslant n$.
Lemma \ref{lemma2} thus asserts that an
almost-round, not round curve admits at
least one bending point.
Given a simple braid $s$, and a
bending point $j$ for some
curve $\mathcal C$, $j$ will be said to
be \emph{compatible} with $s$ if the strands
$j$ and $j+1$ of $s$ do not cross in $s$.

The following is the key for Theorem \ref{theoremartin}:
\begin{lem}\label{mainartin}
Let $s_1,s_2$ be two simple braids such
that $s_1\cdot s_2$ is in left normal form,
and let $\mathcal C$ be a curve.
Let $j$ be a bending point for $\mathcal C$,
compatible with $s_1$. Then there exist some
bending point $j'$ for $\mathcal C^{s_1}$,
which is compatible with $s_2$.
\end{lem}

In order to prove this lemma, we recall
that simple braids are those positive braids in which
any pair of strands crosses at most
once (see~\cite{elrifaimorton}). We also
recall the following well-known fact
(see~\cite{ghys}): any simple braid $s$
can be decomposed in a unique manner as
a product
$$s=\prod_{p=2}^{n}\Sigma_{p,k_p},$$
where $0\leqslant k_p<p$
and $\Sigma_{p,0}=1$.
This allows us to see $s$ as
a sequence of moves of the strands
numbered from 2 to $n$ (in this order),
each of them moving $k_p$ positions to the
left (the number~$k_p$ depending on $s$).
Notice that the $p$th strand of $s$ does not
cross (in $s$) any of the strands being to the
left of the position $p-k_p$.

{\textit{Proof of Lemma \ref{mainartin}.}
As $s_1\cdot s_2$ is in left normal form,
it is sufficient to show that there exist
some bending point $j'$ for $\mathcal C^{s_1}$
such that the strands numbered $j'$ and $j'+1$
at the end of $s_1$ do not cross in $s_1$
(\cite{elrifaimorton}).

Let
$$s_1=(\prod_{p=2}^{j-1} \Sigma_{p,k_p})\cdot\Sigma_{j,k_j}\cdot\Sigma_{j+1,k_{j+1}}\cdot(\prod_{p=j+2}^{n} \Sigma_{p,k_p})$$
be
the decomposition of $s_1$ in terms of
the braids $\Sigma_{p,k}$.
Acting on $\mathcal C$ by $s_1$ following
the above factorization, we will be able
to find at each step a bending point $\iota$
for the resulting curve such that the strands
at respective positions $\iota$ and $\iota+1$
at the end of this step have not yet
crossed in~$s_1$.

The first factor $\prod_{p=2}^{j-1} \Sigma_{p,k_p}$
involves moves of punctures which are to the
left of $j$. Hence, due to Formulae ($F1$) and ($F3$)
and Lemma \ref{reductionartin}, $j$ is still
a bending point for the resulting curve.
Moreover strands ending at $j$ and $j+1$
have not crossed in this first factor.

The factor $\Sigma_{j,k_j}$ moves the $j$th
puncture to the left; once again due to Formulae
($F1$) and ($F3$) and Lemma \ref{reductionartin},
the resulting reduced word admits a
subword $i'\smallfrown j\smallsmile$ (with $i'<j$); $j$
is still a bending point for the corresponding curve.
Notice that the former $j$th puncture lies
(now at position $j-k_j$) under the arc $i'\smallfrown j$
and that strands ending at positions $j$
and $j+1$ have not crossed.

Then, the factor $\Sigma_{{j+1},k_{j+1}}$
moves the $j+1$st puncture to the left.
As $j$ was a bending point for $\mathcal C$,
compatible with $s_1$, the strands which started at positions $j$ and $j+1$ do not cross in $s_1$, hence this movement of the $j+1$st puncture cannot
exceed the position $j-k_j$: it ends at
the diameter of the upper arc $i'\smallfrown j$
and by Formula ($F4$) and Lemma \ref{reductionartin},
the position $\iota_{j+1}:=j-k_{j+1}$ is now a
bending point for the resulting curve
(see Figure \ref{figureclaim3}).
Moreover by construction the strands
ending at $\iota_{j+1}$ and $\iota_{j+1}+1$
have not crossed in $s_1$.
\begin{figure}[ht]
\centerline{\includegraphics{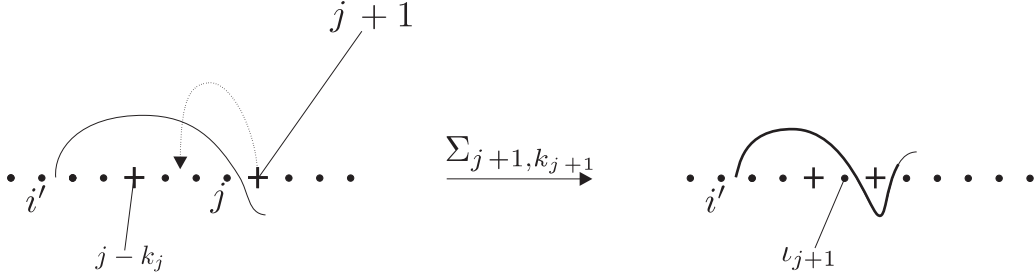}}
\caption{The action of the braid $\Sigma_{j+1,k_{j+1}}$;
as the strands represented as bold crosses
cannot cross, the point  $j+1$ is ``blocked'' in its move to the left by the position $j-k_j$.}
\label{figureclaim3}
\end{figure}

Finally, observe that for $j+2\leqslant q\leqslant n$,
the movement of the $q$th puncture yields
some bending point $\iota_q$ for the resulting curve:
this is $\iota_q=\iota_{q-1}+1$ if this
movement ends to the left of $\iota_{q-1}$
and $\iota_q=\iota_{q-1}$ otherwise.
Thus $j':= \iota_n$ is a bending point
of $\mathcal C^{s_1}$ and strands
numbered $j'$ and $j'+1$ at the end
of $s_1$ do not cross in $s_1$ by construction.
This shows that the bending point $j'$
of $\mathcal C^{s_1}$ is compatible with $s_2$. \hfill $\Box$

We can now complete the proof of
Theorem \ref{theoremartin}.
Recall that we consider a
braid $x=\Delta^px_1\ldots x_r$ in left
normal form which sends a
round curve $\mathcal C$ to another round
curve $\mathcal C'=\mathcal C^{x}$.
Since the braid~$\Delta$ corresponds to a
rotation of $D_n$,
the curve $\mathcal C^{\Delta^t}$ is round
for all integers
$t$ and we may suppose, up to multiplication
by $\Delta^{-p}$, that the left normal form
of $x$ is just $x_1\ldots x_r$.
In order to prove Theorem \ref{theoremartin},
it is sufficient to prove that the curve
$\mathcal C^{x_1}$ is round,
since once this fact is proven, we are given a
braid $x_2\ldots x_r$ in
left normal form whose action transforms
the round curve
$\mathcal C^{x_1}$
into the round curve $\mathcal C'$, and the
result follows by
induction on the number of factors in the
left normal form of $x$.
We shall give a proof by contradiction,
assuming that the curve
$\mathcal C^{x_1}$ is not round.
However it is almost round, as $x_1$ is
a simple braid, see \cite{bertjuan}.

According to Lemma \ref{lemma2},
there must exist a bending
point $j_1$ for $\mathcal C^{x_1}$.
Notice that punctures which lie above
the curve
$\mathcal C^{x_1}$ come from strands
which started to the
right of $\mathcal C$
at the beginning of $x_1$, and 
punctures which
lie below $\mathcal C^{x_1}$
come from strands which started to the
left of $\mathcal C$.
By Lemma~\ref{lemma2}, the puncture
$j_1$ lies either below the
curve~$\mathcal C^{x_1}$ or is
enclosed by it whereas the puncture $j_1+1$
lies either above, or is
enclosed by~$\mathcal C^{x_1}$.
Moreover, at most one of them is enclosed
by~$\mathcal C^{x_1}$.
This implies that the strands
in positions $j_1$ and $j_1+1$ at the end
of $x_1$ do not cross in $x_1$.
In other words, $j_1$ is a bending point
for $\mathcal C^{x_1}$ compatible
with $x_2$ (as $x_1\cdot x_2$ is in left normal form).

Now, by induction on $m=1,\ldots, r-1$
and by Lemma \ref{mainartin}, there
exist $j_2,\ldots,j_{r-1}$ such that $j_m$
bends the curve $\mathcal C^{x_1\ldots x_m}$
and is compatible with $x_{m+1}$,
for all ${m=1,\ldots, r-1}$.
The proof of Lemma \ref{mainartin} also ensures
that $\mathcal C^{x_1\ldots x_r}$ has a bending point $j_r$.

In particular we have shown that the
curve $\mathcal C^{x_1\ldots x_r}=\mathcal C^x$
is not round, provided that $\mathcal C^{x_1}$
is not round.
This is a contradiction;
therefore $\mathcal C^{x_1}$ is round
and Theorem~\ref{theoremartin} is shown.

\section{The dual case, proof of Theorem \ref{theorembkl}}\label{bklproof}
In this section we shall be interested
in the dual (or BKL) Garside structure
of the braid group \cite{bkl}. We thus adopt
the parametrization $D_n^{\ast}$ of the $n$-times
punctured disk defined in the Introduction.
We shall associate to each isotopy class of
simple closed curves $[\mathcal C]$ in $D_n^{\ast}$,
a unique reduced word $W([\mathcal C])$ to
be defined in Subsection~\ref{fromcurvestowordsdual}.
Reduced words will be a tool for describing
the action of dual positive braids on simple
closed curves in $D_n^{\ast}$ in
Subsection \ref{actionofbraidsdual}.
Finally, Theorem \ref{theorembkl}
will be proved in Subsection \ref{thelastone}.

\subsection{From curves in $D_n^{\ast}$ to reduced words}\label{fromcurvestowordsdual}
As in the previous section,
all the curves we are considering are
nondegenerate, simple and closed. We
shall speak about ``curves" when we
really mean ``nondegenerate simple closed curves".

{\textbf{Notation}} : the circle of
radius $\frac{1}{2}$ in $D_n^{\ast}$
centered at the origin will be
denoted by~$\Gamma$. Given $i,j\in\{1,\ldots,n\}$,
the move of the puncture $i$ clockwise
along $\Gamma$ up to the position~$j$
describes an arc of~$\Gamma$
which we will denote by $(i,j)$.
The arc $(i,i)$ is just the puncture $i$.

Let $\mathcal C$ be a curve in $D_n^{\ast}$,
and suppose that it has a finite number of
crossings with the circle $\Gamma$, which
are all transverse. We shall associate
to $\mathcal C$ a word $W(\mathcal C)$ that
we now define.

Choose a point $\ast$ of $\mathcal C$ which
lies on the circle $\Gamma$ and choose an
orientation for~$\mathcal C$. Running
along~$\mathcal C$ following the chosen
orientation, starting and ending at $\ast$,
determines a word in the
alphabet $Y=\{\smallsmile,\cvl,\cvr,\cll,\clr,1,\ldots,n\}$
as follows.\\
Each arc through the inner component of $D_n^{\ast}-\Gamma$
contributes a letter $\smallsmile$ to the word.
Each intersection of $\mathcal C$ with the
arc $(i,i+1)$ of $\Gamma$ contributes a letter $i$.
Finally arcs of $\mathcal C$ through
the outer component of $D_n^{\ast}-\Gamma$
contribute a letter
$\cvr$, ($\cvl$ respectively) if they
are oriented clockwise (counterclockwise, respectively);
except those as in
Figure \ref{examplecurvebkl} (a)
(thus having their endpoints in the same
arc $(i,i+1)$ of $\Gamma$) which contribute
a letter $\clr$ or $\cll$, in a a natural
way according to their orientation.

\begin{figure}[ht]
\centerline{\includegraphics{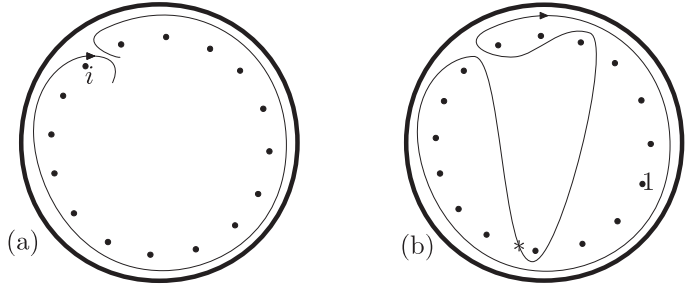}}
\caption{(a) An arc $i\clr i$.  (b) The curve of Example \ref{exacceptable}}
\label{examplecurvebkl}
\end{figure}
The number corresponding to the
intersection point $\ast$ can be chosen to be
either at the beginning or at
the end of the word.
The word obtained in this way
will be denoted by $W(\mathcal C)$
and we call it the
\emph{word associated to $\mathcal C$}.
Choosing another point~$\ast$ or
putting the letter determined by $\ast$
at the beginning or at the end of~$W(\mathcal C)$
corresponds to a cyclic permutation
of the letters in $W(\mathcal C)$,
whereas choosing the
reverse orientation of $\mathcal C$
yields the reverse of $W(\mathcal C)$,
exchanging with each other the
letters $\cvr$ and $\cvl$ ($\clr$ and $\cll$,
respectively). Hence, words associated to
curves are to be considered up to cyclic
permutation of their letters and up to
reverse, exchanging the orientation of
outer arcs.

\begin{ex}\rm
\label{exacceptable}
Let $\mathcal C$ be the curve depicted in
Figure \ref{examplecurvebkl} (b); here we have $n=16$.
The point $\ast$ and the orientation are
also indicated in the figure. This curve yields the word
$$W(\mathcal C)=
4\curvearrowleft3\smallsmile13\curvearrowleft12\smallsmile10\clr10
\smallsmile.$$
\end{ex}

Notice that two curves related by an
isotopy of $D_n^{\ast}$ fixing the
circle $\Gamma$ setwise have the same
associated word.

We say that the word associated
to $\mathcal C$ is \emph{reduced} if it
does not contain any subword
of the form $i\smallsmile i$, $i\cvr i$
or $i\cvl i$. We say that a curve $\mathcal C$
is \emph{reduced} if its associated word
$W(\mathcal C)$ is reduced.
Notice that reduced curves are exactly those
which do not bound any bigon (see \cite{fgrrw})
together with the circle $\Gamma$. According
to \cite{fgrrw}, every curve~$\mathcal C$ is
isotopic to a reduced
one~$\mathcal C^{red}$, which is unique up to
isotopy of $D_n^{\ast}$ fixing the
circle $\Gamma$ setwise. We
finally define, for each isotopy class of
curves $[\mathcal C$] in $D_n^{\ast}$, its
associated reduced word
as $W([\mathcal C]) = W(\mathcal C^{red})$.

\subsection{The action of dual positive braids}\label{actionofbraidsdual}
Let $B_n^{+\ast}$ be the submonoid of $B_n$
generated by the braids
$a_{i,j},\  1\leqslant i,j\leqslant n$, $i\neq j$,
called the monoid of dual positive braids.
We shall use a bigger generating set, namely

$$\mathcal P=\left\{ a_{i_1,i_2}a_{i_2,i_3}\ldots a_{i_{r-2},i_{r-1}}a_{i_{r-1},i_r},
  \left|
  \begin{array}{c}
  2\leqslant r\leqslant n,\\
   i_1,\ldots,i_r\  \text{all distinct and placed in this order,}\\
    \text{following the circle $\Gamma$ clockwise from $i_1$ to $i_r$}
   \end{array}
   \right.
   \right\}.$$
The elements of $\mathcal P$ are naturally called
\emph{polygons} according to their geometric
representation in~$D_n^{\ast}$; and they
correspond to a counterclockwise rotation of a
neighborhood of the convex polygon in~$D_n^{\ast}$
whose vertices are $i_1,i_2,\ldots,i_r$, following
the cyclic permutation $[i_r,i_{r-1},\ldots, i_2,i_1]$.

The following are well-known (see \cite{bkl}):
\begin{itemize}
\item Let $P_1,P_2\in \mathcal P$, seen as
polygons in $B_n^{\ast}$. If their respective
convex hulls are disjoint, then $P_1P_2=P_2P_1$.
In this situation, we will say that $P_1$
and $P_2$ are \emph{disjoint}.
\item Let $i_1,\ldots,i_r$
be $r$ ($2\leqslant r\leqslant n$)
punctures placed in this order
following $\Gamma$ clockwise from $i_1$
up to $i_r$. Then all the braid words
obtained as the concatenation of $r-1$
consecutive letters taken from the sequence
$(a_{i_1,i_2},a_{i_2,i_3},\ldots, a_{i_{r-2},i_{r-1}},a_{i_{r-1},i_r},a_{i_r,i_1})$
in this order, up to cyclic permutation,
are representatives of the same braid~$P$,
which is an element of~$\mathcal P$. Moreover,
for each pair $1\leqslant d<e\leqslant r$, the
letter $a_{i_d,i_e}$ is a prefix of~$P$.
\end{itemize}

Let us now consider the isotopy class
of a curve $\mathcal C$ and its associated
reduced word $W=W([\mathcal C])$.
Given a dual positive
braid $x$, we are going to describe which
transformations have to be
performed on the word $W$ in order to obtain the word
$W([\mathcal C]^x)$. To this purpose we first
focus on the case $x=P\in \mathcal P$.

We assume that the action of $P$ will mainly
modify the inner arcs whereas the outer arcs
are only modified by shifting their endpoints
along the circle $\Gamma$.

We first observe that any inner
arc $i\smallsmile j$ separates
the punctures into
two disjoint subsets: one
containing the punctures $i+1$ and $j$,
and the other containing the
punctures $j+1$ and $i$.
We say that $P$ is \emph{disjoint}
from the arc $i\smallsmile j$ if all
the vertices of $P$ lie in only one of
these sets. If it is so, we set
$(i\smallsmile j)^P=i\smallsmile j$.

Otherwise we say that $P$ is \emph{transverse}
to the arc $i\smallsmile j$.
In that case, let $p_{i,j}$ be the rightmost
vertex of~$P$ lying in $(j+1,i)$;
similarly let $q_{i,j}$ be the rightmost
vertex of~$P$ lying in $(i+1,j)$. By abuse
of notation, we will write $i\in P$ to mean
that $i$ is a vertex of $P$.
Observe that $p_{i,j}=i\Leftrightarrow i\in P$
and $q_{i,j}=j\Leftrightarrow j\in P$.
A priori, looking at the pictures, we would define:
$$
\begin{array}{lclr}
(i\smallsmile j)^P & = &
i\smallsmile p_{i,j}\cvl (p_{i,j}-1)\smallsmile (q_{i,j}-1)\cvr q_{i,j}\smallsmile j & \ \ \ \ \text{($F'$)}
\end{array}
$$

Formula ($F'$) is depicted in Figure \ref{figureformulbkl}.
\begin{figure}[hb]
\centerline{\includegraphics{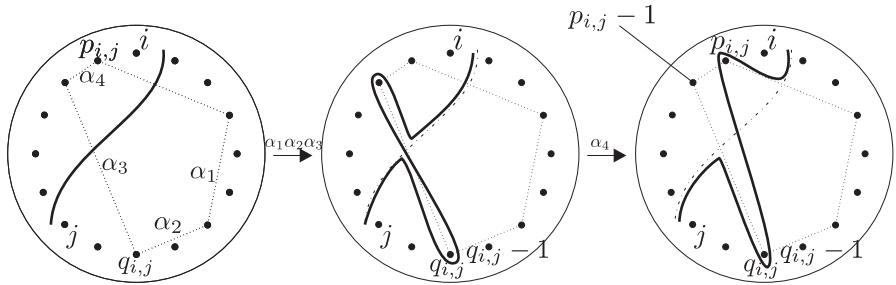}}
\caption{In this example $P$ (in dotted lines) can be expressed as the product $\alpha_1\ldots\alpha_4$; $\alpha_1\alpha_2$ acts trivially on the arc $i\smallsmile j$, $\alpha_3$ yields the arc in the middle part, and the action of $\alpha_4$ on the latter yields the arc in the right-hand side, which is $(i\smallsmile j)^P$, the image under the action of $P$ of the arc $i\smallsmile j$.}
\label{figureformulbkl}
\end{figure}

But notice that application of Formula ($F'$)
produces bigons with the circle $\Gamma$
if either $i\in P$ or $j\in P$. That is why we set:
$$
(i\smallsmile j)^P=
\left\{
\begin{array}{llc}
i\smallsmile p_{i,j}\cvl (p_{i,j}-1)\smallsmile (q_{i,j}-1)\cvr q_{i,j}\smallsmile j & \text{if $i\notin P$ and $j\notin P$}& \text{($F'0$)}\\
i\smallsmile p_{i,j}\cvl (p_{i,j}-1)\smallsmile (j-1) & \text{if $i\notin P$ and $j\in P$}& (\text{$F'1$})\\
(i-1)\smallsmile (q_{i,j}-1)\cvr q_{i,j}\smallsmile j & \text{if $i\in P$ and $j\notin P$} & (\text{$F'2$})\\
(i-1)\smallsmile (j-1) &\text{if $i\in P$ and $j\in P$} & (\text{$F'3$})
\end{array}
\right.$$

Later, we shall need the following:
\begin{remark}\rm\label{remark}
The image of an inner arc $i\smallsmile j$
under the action of a polygon $P$ lies
(up to deformation) in a neighborhood of
the union of $i\smallsmile j$ with $P$.
\end{remark}

Let us replace each
subword $(i\smallsmile j)$ in $W$ by the
corresponding subword $(i\smallsmile j)^P$
as defined above.
This transforms the outer arcs in~$W$ by shifting
their endpoints along the circle $\Gamma$.
Moreover, letters~$\cvr$ ($\cvl$, respectively)
need to be transformed into $\clr$ ($\cll$, respectively) if
they correspond in~$W$ to an arc $(c+1)\cvr c$,
where $c+1$ is shifted up to the position $c$
(or if they correspond to an arc $(c-1)\cvl c$,
where~$c$ is shifted up to the position $c-1$,
respectively). See Figure \ref{tototo}~(a).
Similarly, letters $\clr$ ($\cll$, respectively)
need to be transformed into $\cvr$ ($\cvl$,
respectively) if the endpoints of the arc do
not coincide any more after the suitable
translation. Let us denote by~$\widetilde W$
the word on~$Y$ obtained in this way.
Notice that $\widetilde W$ is not necessarily
reduced, so that $W([\mathcal C]^P)$ need not
be the same as
$\widetilde W$.
The following is the analogue of
Lemma \ref{reductionartin} in the dual setting:

\begin{lem}
\label{keylemmabkl}
Let $[\mathcal C]$ be an isotopy class of curves, and
$W=W([\mathcal C])$. Let $P\in \mathcal P$ be a polygon.
Let~$\widetilde W$ be as above, and let $W^P$ be the word obtained from $\widetilde W$ by removing all instances of subwords of the form $(p-1)\cvr (p-1)\smallsmile$ (or $(p-1)\cvl (p-1)\smallsmile$) where $p$ is a vertex of $P$ whereas $p-1$ is not.
Then we have $W^P=W([\mathcal C]^{P})$.
\end{lem}
\begin{proof}

We observe that
formulae defining $(i\smallsmile j)^P$ do not
contain any subword of the form $c\smallsmile c$
nor $c\smallfrown c$
(by $\smallfrown$ we mean
either $\cvr$ or $\cvl$). Thus the possible
instance of such a subword in $\widetilde W$
necessarily arises from the transformation
of an outer arc of $W$. According to the
formulae above, the translations of punctures
involved in such a transformation turn the
extremities $c$ and $d$ of an outer arc
into $k\in\{c,c-1\}$ and $l\in \{d,d-1\}$
respectively.

First notice that transformations of outer
arcs of the form
$c\clr c$ or $c\cll c$ cannot yield subwords
of the form $c\smallfrown c$.
Now, an arc $c\smallfrown d$ of $W$
(thus with $c\neq d$) will produce a bigon
with the circle $\Gamma$ only if the
above $k$ and $l$ are the same and $|c-d|=1$.
There are two possibilities (up to reverse),
shown in Figure \ref{tototo}.

\begin{figure}[ht]
\centerline{\includegraphics{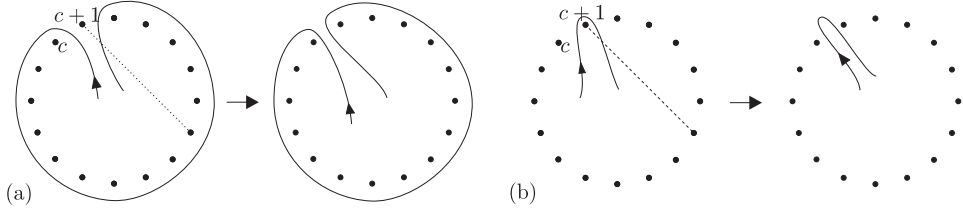}}
\caption{Transformation of an outer arc $c\smallsmile d$ of $W$, with
$|c-d|=1$, into an outer arc of $\widetilde W$ having the same extremities (in  dashed line is depicted an edge of~$P$).
(a) The transformation $c\cvl (c+1)\rightsquigarrow c\cll c$ was already mentioned: no bigon is formed.
(b) The only way (up to reverse) to get a bigon: $c\cvr (c+1) \rightsquigarrow c\cvr c$.}
\label{tototo}
\end{figure}

Finally, by the formulae
defining $(i\smallsmile j)^P$, a necessary
condition for a transformation as in
Figure~\ref{tototo}~(b) to happen is
that $c+1$ is a vertex of $P$ whereas $c$ is not.
We have shown in particular that no subword of
the form $c\smallsmile c$ can arise in $\widetilde W$,
and that the only subwords of the form $c\smallfrown c$
which possibly arise are $(p-1)\smallfrown (p-1)$,
where $p$ is a vertex of $P$ whereas $p-1$ is not.

We now claim that removing all the instances of
these subwords is sufficient in order to
turn $\widetilde W$ into a reduced word; that is,
every sequence $(p-1)\smallfrown (p-1)$
in $\widetilde W$ is a subsequence of a larger
one of the form
$${a\smallsmile (p-1)\smallfrown (p-1)\smallsmile b},$$
with $a\neq b$.
Let $r\smallsmile (p-1)\cvr p\smallsmile v$ be
a subword of $W$ to be transformed into the above
one under the action of $P$ (hence $P$ is transverse
to the arc $p\smallsmile v$). See Figure \ref{figurekeylemma}.

\begin{figure}[ht]
\centerline{\includegraphics{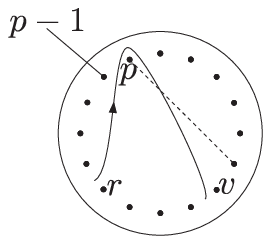}}
\caption{The arc $r\smallsmile (p-1)\cvr p\smallsmile v$ and an edge of $P$ in dashed line.}
\label{figurekeylemma}
\end{figure}

We have (since $p-1$ is not a vertex of $P$)
$$(r\smallsmile (p-1))^P=
\begin{cases}
r\smallsmile (p-1)&\text{if $P$ disjoint from $r\smallsmile (p-1)$},\\
\ldots (q_{r,p-1}-1)\cvr q_{r,p-1}\smallsmile (p-1)& \text{otherwise},
\end{cases}$$
and
$$(p\smallsmile v)^P=
\begin{cases}
 (p-1)\smallsmile (q_{p,v}-1)& \text{if $v\in P$},\\
 (p-1)\smallsmile (q_{p,v}-1)\cvr q_{p,v}\smallsmile v &
 \text{if $v\notin P$}.
\end{cases}$$
Thus it remains to be proved that
the values $a\in \{r,q_{r,p-1}\}$ and $b=q_{p,v}-1$
are distinct.
Notice that by definition, $q_{p,v}$ lies
in $(p+1,v)$ so that $b=q_{p,v}-1$ lies in $(p,v-1)$.
On the other hand, as $\mathcal C^{red}$
is simple and $W$ is its associated reduced
word, $r\in (v,p-2)$
(see Figure \ref{figurekeylemma}).
Finally by definition, $q_{r,p-1}$
lies in $(r+1,p-1)\subset (v+1,p-1)$.
Hence in any case,
$a\in(v,p-1)$;
this shows that $a\neq b$ and achieves
the proof of Lemma \ref{keylemmabkl}.
\end{proof}

We thus can associate to each isotopy class of
curves $[\mathcal C]$ and each braid $P\in \mathcal P$
the word $W([\mathcal C])^{P}$ defined thanks to
the above construction.
We are now able to define, for each
isotopy class of curves $[\mathcal C]$, the image
of its associated reduced word ${W=W([\mathcal C])}$
under the action of some dual positive braid $x$.
Indeed, if $x$ is expressed as a
product $x=\prod_{m=1}^{r} P_m$ (where each factor
lies in~$\mathcal P$), then by Lemma~\ref{keylemmabkl},
the inductive formula
$$W^x=
(W^{\prod_{m=1}^{r-1} P_m})^{P_{r}}$$
defines a word on $Y$ which is the reduced word associated to
$[\mathcal C]^x$ (hence
it does not depend on
the chosen decomposition of $x$ in terms of braids in $\mathcal P$).
This can be written
${W({[\mathcal C]}^x)=W([\mathcal C])^x}$.

\subsection{Proof of Theorem \ref{theorembkl}}\label{thelastone}
By abuse of notation, we will
speak about curves instead of isotopy classes
of curves, meaning that we are always considering
the reduced representatives. Consequently the
letter $\mathcal C$
will denote the reduced representative of the
isotopy class of the curve $\mathcal C$ and its
associated reduced word will be denoted
by $W(\mathcal C)$.\\
We shall now prove the analogue of
Theorem \ref{theoremartin} in the dual setting.
As mentioned in the introduction,
the statement of Theorem \ref{theoremartin},
with round curves defined as circles surrounding a
set of consecutive punctures is false in this
setting, as shows the following example:
\begin{ex}\label{exampleprime}\rm
Let $n=4$. Consider the braid $x=a_{1,2}.a_{1,4}$
which is in left normal form as written.
Figure \ref{exampleroundness} shows that roundness
is not preserved after each factor
of the left normal form if we define it as the
property
of being homotopic to a geometric circle surrounding
a set of consecutive punctures.
\begin{figure}[h]
\centerline{\includegraphics{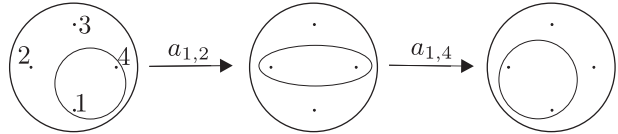}}
\caption{The curve in the middle part fails to be homotopic to a circle surrounding a set of consecutive punctures, although it is the image, under the first factor of the left normal form of $x$, of a
``round" curve which is sent by $x$ to a ``round" curve.}
\label{exampleroundness}
\end{figure}
\end{ex}
We then need to define a suitable class of curves,
which will play the role played by round curves
in the classical setting.
It is the following:
\begin{defi}\label{defistandard}
A curve $\mathcal C$ will be called \emph{standard} if $W(\mathcal C)$ only
admits letters in $\{1,\ldots,n,\cvl,\smallsmile\}$
(or in $\{1,\ldots,n,\cvr,\smallsmile\}$, up to reverse).
\end{defi}
Notice that this is equivalent to
Definition \ref{informaldefistandard},
in the Introduction. It turns out that
Theorem \ref{theoremartin}
holds in the dual setting, if we replace round
curves by \emph{standard} curves. This is the
statement of Theorem \ref{theorembkl}, which we
will now prove.

First, we recall from \cite{bkl} the following facts:
\begin{claim}\label{claim}
Every dual simple braid can be written in a
unique manner (up to permutation of the factors)
as a product of pairwise disjoint elements
of $\mathcal P$ (hence a commutative product).
\end{claim}
Therefore, by Remark \ref{remark}, studying the
action of dual simple braids on the arc~${i\smallsmile j}$
of the curve~$\mathcal C$ through the decomposition $s=P_1\ldots P_g$
can be done quite easily, since all inner arcs
obtained by application of Formulae ($F'$) for some $P_l$
are invariant under the action of $P_k, k\neq l$.
See Figure~\ref{bigbkl}.

\begin{lem}\label{obstructing}\cite{bkl}
\begin{enumerate}
\item[1)]If $s$ is a dual simple braid and $s=P_1\ldots P_t$ is its decomposition into pairwise disjoint factors in $\mathcal P$, then
$sa_{i,j}$ is simple if and only if $P_ua_ {i,j}$ is simple for every~${u=1,\ldots,t}$.
\item[2)]If $P\in \mathcal P$, then $Pa_{i,j}$ is simple if and only if no word representative of $P$ can be written with a letter $a_{k,l}$ such that $k\in (j+1,i)$ and $l\in(i+1,j)$.
\end{enumerate}
\end{lem}
We introduce, as in the previous section,
some further notation. Given a curve $\mathcal C$
in $D_n^{\ast}$, and an unordered pair $i,j$ of
punctures, we say that $i,j$ is a \emph{bending pair}
for $\mathcal C$ if the reduced word $W(\mathcal C)$
contains a sequence of the form $\cvr i\smallsmile j\cvl$ or $\cvr i\smallsmile j\cll$
(up to reverse). Also, given a dual simple element $s$,
a bending pair $i,j$ for a curve $\mathcal C$ will
be said to be \emph{compatible} with $s$ if $a_{i,j}$
is not a prefix of $s$.

We are now able to state and prove the key lemma,
aiming to Theorem \ref{theorembkl}; it is the
analogue of Lemma \ref{mainartin} in the dual setting:
\begin{lem}
\label{proofbkl}
Let $s_1,s_2$ be two dual simple braids such
that $s_1\cdot s_2$ is in dual left normal form.
Let~$\mathcal C$ be a curve and assume that $\mathcal C$
admits some bending pair $i,j$, which is compatible
with $s_1$. Then there exist a bending pair $i',j'$
for $\mathcal C^{s_1}$ which is compatible with $s_2$.
\end{lem}

\begin{proof}
Consider the decomposition $s_1=P_1\ldots P_g$
of $s_1$ into pairwise disjoint elements of $\mathcal P$.
Notice that, as $s_1\cdot s_2$  is in left
normal form, $a_{i',j'}$ is not a prefix of $s_2$
whenever $s_1a_{i',j'}$ is simple, thus it is sufficient
to find a bending pair $i',j'$ for $\mathcal C^{s_1}$
such that $s_1a_{i',j'}$ is simple.

By hypothesis, $a_{i,j}$ is not a prefix of $s$. Thus,
none of the polygons $P_1,\ldots,P_g$ has both $i$
and $j$ as vertices.
First, the action of any polygon
 among $P_1,\ldots,P_g$ which is disjoint
 from the arc $i\smallsmile j$ in $W(\mathcal C)$
 results in a curve for which $i,j$ is still a
 bending pair (see Lemma~\ref{keylemmabkl}).
Moreover, by Lemma~\ref{obstructing}, if $P$ is
such a polygon, then $Pa_{i,j}$ is simple.
\begin{figure}[ht]
\centerline{\includegraphics{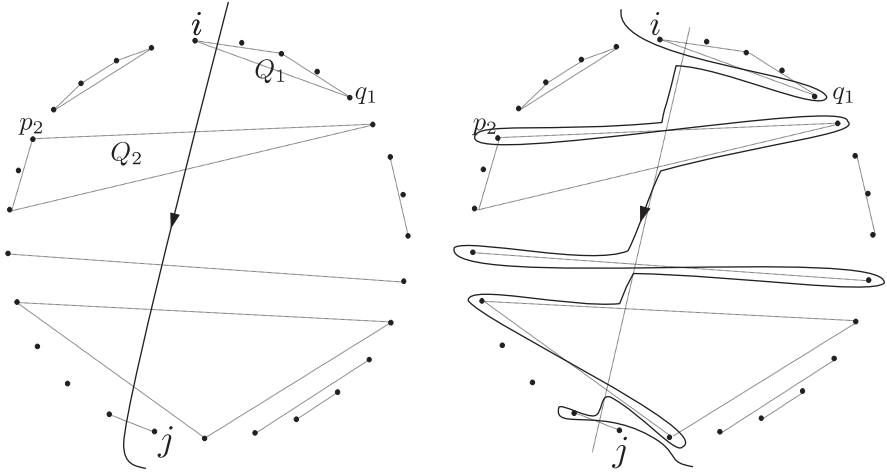}}
\caption{The braid $s_1$, as a product of pairwise disjoint polygons, is depicted in dashed lines. We can see, from left to right, the action of this braid on the arc $i\smallsmile j$.}
\label{bigbkl}
\end{figure}

Then, the action of the polygons
among $P_1,\ldots,P_g$ which are transverse
to the arc $i\smallsmile j$ can be studied as
follows (see Figure \ref{bigbkl}).
The involved polygons can be ordered by running
along the arc $i\smallsmile j$ starting
at $(i,i+1)$: $Q_1,\ldots,Q_h$.
For ${t=1,\ldots, h}$, if $p_t$ is the
rightmost vertex of $Q_t$ in $(j+1,i)$ and $q_t$
is the rightmost vertex of $Q_t$ in $(i+1,j)$, then
by Formulae ($F'$) and Lemma \ref{keylemmabkl}, the pair
$$i',j'=\begin{cases}
q_1,p_2  & \text{if $h>1$}\\
i,p_1    & \text{if $h=1$ and $i$ is not a vertex of $Q_1$}\\
q_1,j      & \text{if $h=1$ and $j$ is not a vertex of $Q_1$}
\end{cases}$$
is a bending pair for the curve $\mathcal C^{s_1}$.
Moreover, by Lemma \ref{obstructing}, in any case
the braid $s_1a_{i',j'}$ is simple. It follows
that $i',j'$ is a bending pair for the
curve $\mathcal C^{s_1}$ compatible with $s_2$.
\end{proof}

We now complete the proof
of Theorem \ref{theorembkl}.
Recall that we consider a
braid ${x=\delta^px_1\ldots x_r}$ in dual left normal
form which sends some standard curve $\mathcal C$
to another standard curve $\mathcal C'=\mathcal C^x$.
Since the braid $\delta$ corresponds to a
rotation of $D_n^{\ast}$, and thus sends
standard curves to standard curves,
up to multiplication by a power of $\delta$
we may assume that~$x$ is a dual positive braid,
whose left normal form is $x_1\ldots x_r$.
By a direct induction on the number of factors
in the left normal form of $x$, it is sufficient
to show that~$\mathcal C^{x_1}$ is standard.
We shall give a proof by contradiction,
assuming that~$\mathcal C^{x_1}$ is nonstandard.
We will see (by induction on
$m=1,\ldots,r$) that none of the curves
$\mathcal C^{x_1\ldots x_m}$
for $1\leqslant m\leqslant r$ can be standard,
contradicting the fact that $\mathcal C^x$ is standard.

Let $x_1=P_1\ldots, P_g$ be the decomposition
of $x_1$ into pairwise disjoints elements
of $\mathcal P$.
On the other hand suppose that $W(\mathcal C)$
is written only with letters in $\{\smallsmile,\cvl,1,\ldots,n\}$
(that is choose the counterclockwise
orientation for $\mathcal C$).

As $\mathcal C^{x_1}$ is nonstandard,
so must be $\mathcal C^{P_t}$ for at
least one of the~$P_t$'s (according to
Remark \ref{remark}).
We may assume that $\mathcal C^{P_1}$ is
nonstandard.
Therefore, there exist an arc $i\smallsmile j$
in $W(\mathcal C)$ such that $P_1$ is transverse
to $i\smallsmile j$ and $j$ is not a vertex of $P_1$
(hence $W(\mathcal C^{P_1})$ has at least one
letter $\cvr$, see Formulae ($F'0$) and ($F'2$)).

Consider all the polygons among $P_1,\ldots,P_g$
which are transverse to the arc $i\smallsmile j$.
By Remark~\ref{remark}, only these polygons
witness the action of $x_1$ on $i\smallsmile j$.
Running along the arc~${i\smallsmile j}$ starting
at $(i,i+1)$ allows us to order them in a natural
way (see the proof of
Lemma \ref{proofbkl}): $Q_1,\ldots,Q_h$.
Let~$q_1$ be the rightmost
vertex of $Q_1$ in $(i+1,j)$ (so that $q_1\neq j$)
and if $h>1$ let~$p_2$ be the
rightmost vertex of $Q_2$ in $(j+1,i)$.

Then we set $i_1=q_1$ and $j_1=p_2$
if $h>1$, $j_1=j$ otherwise.
Formulae ($F'$) and Lemma \ref{keylemmabkl}
now imply that $W(\mathcal C^{x_1})$ contains
the subword $\cvr i_1\smallsmile j_1\cvl$ (or $\cvr i_1\smallsmile j_1\cll$).
By construction, the braid $x_1a_{i_1,j_1}$
is simple according to Lemma~\ref{obstructing}.

In other words we saw that $i_1,j_1$ is a
bending pair for $\mathcal C^{x_1}$, and
since $x_1\cdot x_2$ is in left normal form,
this bending pair is compatible with $x_2$.
It follows by induction on $m$ and
Lemma \ref{proofbkl} that one can find,
for each $m=1,\ldots, r-1$, a bending
pair $i_m,j_m$ for the
curve $\mathcal C^{x_1\ldots x_m}$
which is compatible with $x_{m+1}$.
The existence of a bending pair $i_r,j_r$ for
the curve $\mathcal C^{x_1\ldots x_r}$ also
follows from the proof of Lemma \ref{proofbkl}.

We proved in particular that the
curve $\mathcal C^{x_1\ldots x_r}=\mathcal C^x$
is not standard. This is a contradiction
which completes the proof of Theorem \ref{theorembkl}.


\section{Deciding the dynamical type of braids}
The above results give rise to an algorithm
for deciding the Nielsen-Thurston type of a
given braid, in the spirit of \cite{leelee},
\cite{bertjuan}, using the dual structure.
Thurston's classification Theorem~\cite{thurston}
asserts that the elements of Mapping Class
Groups of surfaces (and therefore, in particular, braids)
split into three mutually exclusive types, according to their
dynamical properties: periodic, reducible non-periodic and
pseudo-Anosov. The reader is referred to \cite{thurston}, \cite{fathi}
for details and precise definitions. We restrict ourselves to recalling
that reducible braids are those preserving a family of
pairwise disjoint isotopy classes of nondegenerate
simple closed curves in the $n$-times
punctured disk.
Periodic braids being easy to detect~\cite{bggmperiodic},
the main problem to be solved is to
decide whether a non-periodic braid is pseudo-Anosov
or reducible.
In what follows, we will assume that the braids under
consideration are not periodic and the curves, simple,
closed and nondegenerate, will be considered up to
isotopy so that the term \emph{curve} will be
applied to the isotopy class of a curve.

From Theorem \ref{theorembkl}, it follows that if a braid
preserves a family of standard curves, then its cyclic sliding
\cite{juanvolkercyclic} also preserves a family of standard curves
(see Proposition 4.2. and Corollary 4.3.
in~\cite{juan2010}). Therefore:

\begin{prop}
Let $x$ be a reducible braid. Then there exist
some $y$ in $SC_{BKL}(x)$
(the set of sliding circuits of $x$ \cite{juanvolkercyclic}, with respect
to the dual structure)
which preserves a family of (pairwise disjoint)
standard curves.
\end{prop}

The main result of this section asserts that
this last condition, i.e. preserving a (non-empty) family
of pairwise disjoint standard curves, is checkable in polynomial
time. In fact, we shall prove:

\begin{thm}\label{Thmalgo}
There is an algorithm which decides whether a given
$n$-strand braid~$x$ in dual left normal form
$x=\delta^p x_1\ldots x_{\ell}$
admits a standard invariant curve.
Moreover this algorithm takes time $O(\ell\cdot n^4)$.
\end{thm}
\begin{proof}
Given a subset $I_0$ of $\{1,\ldots,n\}$ (of cardinality
$2\leqslant \#(I_0)\leqslant n-1$) and a braid $x$
in dual left normal form $x=\delta^px_1\ldots x_{\ell}$,
the main task of the algorithm is to construct a bigger
set $S(I_0,x)$ of punctures, which must be
enclosed by the image under $x$ of any standard curve
$\mathcal C$ surrounding
punctures in $I_0$, provided $\mathcal C^x$ is standard.
This is achieved by the following, which is the
key result:

\begin{lem}\label{Lemalgo}
Suppose we are given a dual simple braid $s$,
decomposed into pairwise disjoint
polygons $s=P_1\ldots P_g$, together
with a proper subset $I$ of $\{1,\ldots,n\}$ of cardinality at least
2, whose elements
are enumerated $p_1,\ldots,p_k$ in this order (up to
cyclic permutation)
running along the circle $\Gamma$ clockwise (we will put $p_{k+1}=p_1$).
For each $i=1,\ldots,k$, consider all the polygons
among $P_1,\ldots, P_g$
having at least one vertex, but not all,
in $(p_{i}+1,p_{i+1}-1)$ and which does not have $p_i$ as a vertex,
and take for each of them
its leftmost vertex in the arc $(p_i+1,p_{i+1}-1)$.
Let $I'$ be the union of~$I$ with all punctures collected
in the above way.
If $\mathcal C$ is a standard curve surrounding the punctures
in $I$ (and possibly other punctures) such that~$\mathcal C^s$ is standard,
then $\mathcal C$ surrounds the punctures
in $I'$. Moreover, if $I'=I$, the standard curve whose set
 of inner punctures is exactly
 $I$ is sent to a standard curve by $s$.
\end{lem}
\begin{figure}[ht]
\centerline{\includegraphics{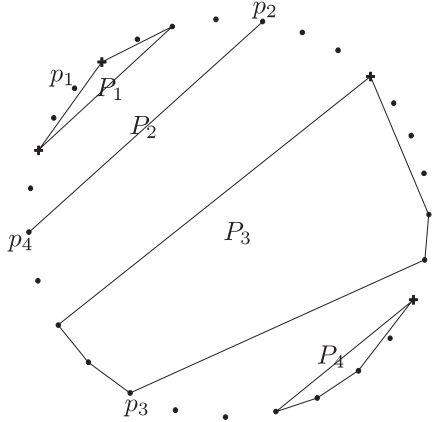}}
\caption{An example that illustrates Lemma \ref{Lemalgo}: here the simple braid $s$ is decomposed as
$s=P_1\ldots P_4$, and $I$ is made of 4 points $p_1,p_2,p_3,p_4$. Then running along each arc $(p_i,p_{i+1})$
allows the construction of $I'$, which consists of adding punctures depicted as crosses to $I$.}
\label{F:algorithm}
\end{figure}

\begin{proof}
First observe that the image of a standard curve under
$s$ is standard if and only if the image of this curve
under each of the $P_i$'s is standard (Remark \ref{remark}).
Now, for a polygon $Q$ and a standard curve $\mathcal C$
oriented counterclockwise, according to Formulae ($F'$),
the following are equivalent:
\begin{itemize}
\item $\mathcal C^Q$ is standard.
\item for each inner arc $a\smallsmile b$ of $\mathcal C$ which is transverse to $Q$, the puncture $b$ is a vertex of $Q$.
\end{itemize}
Let $\mathcal C$ be a standard curve (oriented counterclockwise)
surrounding the punctures in $I$,
such that~$\mathcal C^s$ is standard. The punctures in $I'-I$ (that is the punctures
added to $I$ by the process of Lemma~\ref{Lemalgo})
cannot lie in the outer component of $D_n^{\ast}-\mathcal C$. Indeed, the belonging of such
a puncture to the outer component of $D_n^{\ast}-\mathcal C$ would yield
some inner arc $a\smallsmile b$ of $\mathcal C$, transverse to
a polygon in $s$ not having $b$ as a vertex; this would be in contradiction with the above
remarks. See Figure \ref{F:algorithm}.

Now let $\mathcal C_{I}$ be the standard curve surrounding exactly the
punctures in $I$ (oriented counterclockwise) and suppose that the process of
Lemma \ref{Lemalgo} yields $I'=I$. Then, whenever a
polygon in~$s$ and an inner arc $a\smallsmile b$ of $\mathcal C_{I}$
are transverse, the puncture $b$ is a vertex of the involved polygon
(otherwise some puncture in $(b+1,a)$ would be added to
$I$ by the process of Lemma \ref{Lemalgo}).
By the remark above, it follows that $\mathcal C_{I}^s$ is standard, as claimed.
\end{proof}
The set $I'$ of the above lemma depends only on $I$ and $s$,
and we will
denote by $S(I,s)$ the set $\rho(s)(I')$ (where $\rho$ is the
natural morphism $B_n\longrightarrow S_n$).

Lemma \ref{Lemalgo} says that if $\mathcal C$ is a standard curve
surrounding the punctures in $I$ and $\mathcal C^s$ is standard,
then $\mathcal C^s$ must surround the punctures in $S(I,s)$.
Using Theorem \ref{theorembkl} and an induction on the number of non-$\delta$
factors in the left normal form of $x$ allows to construct the set $S(I_0,x)$.

First, we set $S(I_0,\delta^p x_1)=S(\rho(\delta^p)(I_0),x_1)$.
Then
we define, for $i=1,\ldots,\ell-1$,
$$S(I_0,\delta^px_1\ldots x_{i+1})=S(S(I_0,\delta^px_1\ldots x_{i}),x_{i+1}).$$
Notice that the set $S(I_0,x)$ can be computed in time $O(\ell\cdot n)$. Notice also
that, in virtue of Lemma~\ref{Lemalgo}, the equality $S(I,x)=I$ implies that the curve
whose set of  inner punctures is exactly~$I$ is $x$-invariant.

The last step in the proof of Theorem \ref{Thmalgo}
is the following:
\begin{prop}
Let $a,b$ be any pair of punctures in $\{1,\ldots,n\}$.
There is an algorithm which decides whether a given $n$-braid
$x$ of length $\ell$ admits a standard invariant
reduction curve surrounding the punctures $a$ and $b$. Moreover this algorithm runs in time $O(\ell\cdot n^2)$.
\end{prop}
\begin{proof}
The algorithm does the following:
\begin{itemize}
\item set $I_0=\{a,b\}$,
\item for $m=1,\ldots, n-2$, compute the set
$I_m=S(I_{m-1},x)\cup I_{m-1}$.
\end{itemize}
Remark that $I_{i-1}\subset I_{i}$ for all $i$.
If $I_{n-2}=\{1,\ldots,n\}$, then the algorithm answers negatively;
otherwise, the standard curve which surrounds exactly the punctures
in $I_{n-2}$ is $x$-invariant. Indeed, in the latter case,
there must exist some $k$, ${1\leqslant k<n-2}$, such that
$\#(I_k)=\#(I_{k+1})$, and therefore
$I_k=I_{k+1}=I_{n-2}$. This means that
$S(I_k,x)=I_k$ and therefore
the standard curve whose set of inner punctures
is exactly $I_{n-2}$ (and thus contains $a$ and $b$)
is $x$-invariant. The complexity of the algorithm is
$O(\ell\cdot n^2)$, according to the above estimation about the
computation of $S(I_0,x)$.
\end{proof}
Iterating the above algorithm for each pair of points in
$\{1,\ldots,n\}$ yields the algorithm in the statement of
Theorem \ref{Thmalgo}, since the number of pairs of points
in $\{1,\ldots,n\}$ is $\frac{n\cdot (n-1)}{2}$, so that
the complexity of the whole algorithm is $O(l\cdot n^4)$ as claimed.
\end{proof}
We observe that this algorithm is the analogue of Theorem
2.9. in \cite{bertjuan}. The result in \cite{bertjuan}
deals with the classical structure,
but needs an additional hypothesis about the triviality
of the ``inner" braid; we believe that obtaining the same conclusion
without this hypothesis could be an indication that the dual
structure is more adapted for the kind of problems we are dealing with.

The algorithm deciding the Nielsen-Thurston type of a given braid $x$
is as follows:
\begin{itemize}
\item[1.]Test whether $x$ is periodic \cite{bggmperiodic}, and if it is so, return
``periodic", and stop.
\item [2.]Compute the set of sliding circuits (for the dual structure)
of $x$ \cite{gebhardtjuanalgo}.
\item [3.]For each element $y$ of $SC_{BKL}(x)$,
for each $k=1,\ldots,\frac{n}{2}$, apply the algorithm of
Theorem~\ref{Thmalgo} to the braid $y^k$ (a curve belonging to a family
fixed by $y$ must indeed be fixed
by some power $y^k$, with $k=1,\ldots,\frac{n}{2}$).
\item[4.]Stop whenever a positive answer is found, and return
``reducible"; otherwise return ``pseudo-Anosov".
\end{itemize}
The complexity of this algorithm is not bounded above by a polynomial in
$n$ and $\ell$ since the size of the set of sliding circuits
is known to be exponential in general
\cite{juanvolkercyclic}, \cite{prasolov}.


\begin{thebibliography}{1}
\bibitem{bgn} D. Benardete, M. Gutierrez, Z. Nitecki, \textit{A combinatorial approach to reducibility of mapping classes}, Contemp. Math. 150 (1993), 1-31.
\bibitem{bggmperiodic} J. S. Birman, V. Gebhardt, J. Gonz\'{a}lez-Meneses, \textit{Conjugacy in Garside groups III: Periodic braids}, J. Alg. 316 (2007), 746-776.
\bibitem{bkl} J. Birman, K.H. Ko, S.J. Lee,
\textit{A new approach to the word problem in the braid groups},
Adv. Math. 139 (1998), 322-353.
\bibitem{bertmatth} M. Calvez, B. Wiest, \textit{Fast algorithmic Nielsen-Thurston classification of four-strand braids}, to appear in J. Knot Theory and Ramifications.
\bibitem{dehparis} P. Dehornoy, L. Paris, \textit{Gaussian groups and Garside groups, two generalisations of Artin groups}, Proc. London Math. Soc. (3) 79 (1999), 569-604.
\bibitem{elrifaimorton} E. ElRifai, H. Morton, \textit{Algorithms for positive braids}, Quart. J. Math. Oxford Ser. (2) 45 (1994), 479-497.
\bibitem{echlpt} D.B.A. Epstein, J. Cannon, D. Holt, S. Levy, M. Paterson, W. Thurston, \textit{Word processing in groups}, Jones and Bartlett Publishers, Boston, MA, 1992.
\bibitem{fathi} A. Fathi, F. Laudenbach, V. Poenaru, \textit{Travaux de Thurston sur les surfaces}, Ast\'{e}risque 66-67, SMF 1991/1979.
\bibitem{fgrrw} R. Fenn, M. Greene, D. Rolfsen, C. Rourke, B. Wiest,
\textit{Ordering the braid groups}, Pacific J. Math. 191 (1999), 49-74.
\bibitem{ghys} P. de la Harpe, \textit{An invitation to Coxeter groups}, in E. Ghys, A. Haefliger, A. Verjovsky, \textit{Group theory from a geometrical viewpoint}, World Scientific, 1991.
\bibitem{garside} F. Garside, \textit{The braid group and other groups}, Quart. J. Math. Oxford Ser. (2) 20 (1969), 235-254.
\bibitem{juanvolkercyclic} V. Gebhardt, J. Gonz\'{a}lez-Meneses, \textit{The cyclic sliding operation in Garside groups}, Mathematische Zeitschrift 265 (1),
(2010), 85-114.
\bibitem{gebhardtjuanalgo}
V. Gebhardt, J. Gonz\'{a}lez-Meneses, \textit{Solving the conjugacy problem in Garside groups by cyclic sliding}, J. Symb. Comp. (6) 45 (2010), 629-656.
\bibitem{godelle} E. Godelle, \textit{Parabolic subgroups of Garside groups II: ribbons}, J. Pure Appl. Algebra 214 (2010), 2044-2062.
\bibitem{personalgodelle} E. Godelle, Personal communication.
\bibitem{juan2010} J. Gonz\'{a}lez-Meneses, \textit{On reduction curves and Garside properties of braids}, Contemp. Math. 538 (2011), 227-244.
\bibitem{bertjuan} J. Gonz\'{a}lez-Meneses, B. Wiest, \textit{Reducible braids and Garside theory}, Algebr. Geom. Topol. 11 (2011), 2971-3010.
\bibitem{leelee} E.K. Lee, S.J. Lee, \textit{A Garside theoretic approach to the reducibility problem in braid groups}, J. Algebra 320 no. 2 (2008), 783-820.
\bibitem{prasolov} M. V. Prasolov, \textit{Small braids having a big Ultra Summit Set}, Mat. Zametki, 89 no. 4 (2011), 577-588.
\bibitem{thurston} W. Thurston, \textit{On the geometry and dynamics of diffeomorphisms of surfaces}, B. Amer. Math. Soc. no. 19 (1988), 417-431.
\end{thebibliography}
\end{document}